\documentclass{amsart}

\usepackage{etex}
\usepackage{amsmath, amssymb}
\usepackage{array}
\usepackage[frame,cmtip,arrow,matrix,line,graph,curve]{xy}
\usepackage{graphpap, color, paralist, pstricks}
\usepackage[mathscr]{eucal}
\usepackage[pdftex]{graphicx}
\usepackage[pdftex,colorlinks,backref=page,citecolor=blue]{hyperref}
\usepackage{pifont}

\newtheorem{theorem}{Theorem}[section]
\newtheorem{proposition}[theorem]{Proposition}
\newtheorem{corollary}[theorem]{Corollary}
\newtheorem{lemma}[theorem]{Lemma}
\theoremstyle{definition}

\newtheorem{conjecture}[theorem]{Conjecture}

\newcommand{\gO}{\Omega}

\newcommand{\cA}{\mathcal{A} }

\newcommand{\cM}{\mathcal{M}}

\newcommand{\cR}{\mathcal{R} }

\newcommand{\beq}[1]{\begin{equation}\label{#1}}
\newcommand{\enq}[0]{\end{equation}}

\newcommand{\eps}{\epsilon}

\newcommand{\vp}[0]{\varphi}

\newcommand{\mn}[0]{\medskip\noindent}
\newcommand{\nin}[0]{\noindent}

\newcommand{\sub}[0]{\subseteq}
\newcommand{\sm}[0]{\setminus}
\newcommand{\0}[0]{\emptyset}
\newcommand{\ra}[0]{\rightarrow}

\newcommand{\dist}[0]{\mbox{\rm{dist}}}

\newcommand{\bin}[0]{\mbox{\rm{Bin}}}
\newcommand{\Po}[0]{\mbox{\rm{Po}}}
\newcommand{\pr}[0]{\mathbb{P}}

\newcommand{\R}{\mathcal{R} }
\newcommand{\U}{\mathcal{U} }
\newcommand{\gc}{\gamma}
\newcommand{\ga}{\alpha}
\newcommand{\vs}{\varsigma}

\newcommand{\E}{\mathbb E} 
\newcommand{\C}[2]{{{#1}\choose{{#2}}}}
\newcommand{\m}[0]{\mathcal{M}}

\newcommand{\one}[0]{\textbf{1}}
\newcommand{\h}{\mathcal{H} }

\newcommand{\rad}[0]{~~\mbox{\raisebox{-.05ex}{$\stackrel{\textrm{d}}{\longrightarrow}$}}~~}

\begin{document}

\title{Tuza's conjecture for random graphs}

\author{Jeff Kahn \and Jinyoung Park}
\thanks{JK was supported by NSF Grants DMS1501962 and DMS1954035}
\email{jkahn@math.rutgers.edu, jp1324@math.rutgers.edu}
\address{Department of Mathematics, Rutgers University \\
Hill Center for the Mathematical Sciences \\
110 Frelinghuysen Rd.\\
Piscataway, NJ 08854-8019, USA}

\begin{abstract}
A celebrated conjecture of Zs.\ Tuza says that in any (finite) graph,
the minimum size of a cover of triangles by edges is at most twice the maximum 
size of a set of edge-disjoint triangles. Resolving a recent question of Bennett, Dudek and Zerbib,
we show that this is true for random graphs; more precisely:
\[
\mbox{\emph{for any $p=p(n)$, $\pr(\mbox{$G_{n,p}$ satisfies Tuza's Conjecture})\ra 1 $ (as $n\ra\infty$).}}
\]
\end{abstract}

\maketitle

\section{Introduction}\label{Intro}

In this paper, we use \textit{matching} and \textit{cover} for triangle-matching and cover of triangles
by edges, 
and $\nu$ and $\tau$ for the corresponding matching and cover numbers;
thus, for a (finite) graph $H$, $\nu(H)$ is the maximum size of a set of edge-disjoint triangles 
and $\tau(H)$ is the minimum size of a set $F$ of edges with the property that each
triangle contains a member of $F$.

We are interested in the celebrated \emph{Tuza's Conjecture}:
\begin{conjecture}[Tuza \cite{Tuza}]\label{TC}
For any graph $H$, $\tau(H) \le2 \nu(H)$.
\end{conjecture}
\nin
The inequality is tight when $H$ is $K_4$ or $K_5$ (or, e.g., a disjoint union of copies of
these), and is not far from tight in other cases not related to these two examples
(even if the graph is $K_4$-free; see \cite{HKT}).
We will not survey the literature---see e.g.\ \cite{HKT,BDZ}---and
just mention that the best general result remains that of Haxell~\cite{Haxell}:
\[
\mbox{for every $H$, $~\tau(H)\leq \frac{66}{23}\nu(H)$.}
\]

Here we consider a question raised recently by Bennett, Dudek and Zerbib \cite{BDZ} and
independently by Basit and Galvin \cite{Galvin}; informally: 
\emph{is Tuza's conjecture true for random graphs?}
More precisely, is it true that for any $p=p(n)$ and $G_{n,p}$ the usual binomial
(or ``Erd\H{o}s-R\'enyi'') random graph, 
 w.h.p.\footnote{``with high probability,'' meaning with probability tending to 1 as $n \rightarrow \infty$.}
\[
\mbox{ $G_{n,p}$ satisfies Tuza's Conjecture?}
\]

In \cite{BDZ} this was shown to be true if $p < c_1n^{-1/2}$ or $p>c_2n^{-1/2}$,
with $c_1\approx 0.48$ and $c_2\approx 4.25$.
(They work with $G_{n,m}$, but, as usual, this is about the same as $G_{n,p}$ with $p=m/\C{n}{2}$
and we will stick to the binomial version.)
Here we finish this story:
\begin{theorem}\label{MT}
For any $p=p(n)$, 
\mbox{$\tau(G_{n,p})\leq 2\nu(G_{n,p})~$ w.h.p.}
\end{theorem}
\nin
(This is in some sense a failure:  for a while it seemed to us that the gap in \cite{BDZ}
might hide counterexamples to Tuza's Conjecture.)

We recently heard from
Patrick Bennett that he, Ryan Cushman and Andrzej Dudek \cite{BCD} have
also closed the gap in \cite{BDZ}, using an approach similar to that of the earlier paper
(and different from what we do here).

For the rest of this paper
we use $G$ for $G_{n,p}$, and set $m=\C{n}{2}p$ ($=\E |G|$) and $d=(n-2)p^2$ 
(the expected number of triangles on a given edge of $G$).
Of course for the proof of Theorem~\ref{MT} we could confine ourselves to $d$ in the range
not covered by \cite{BDZ}, but we will give arguments for the full range, in the process
strengthening the earlier results.

To begin, for smallish $d$, we have an asymptotically optimal statement:

\begin{theorem}\label{thm.smalld}
If $d\le 1/2$, then w.h.p.
$ ~
\tau(G)\sim\nu(G).
$
\end{theorem}
\nin
(Of course $\tau(G)\geq \nu(G)$ is trivial.)

For treatment of larger $d =\Theta(1)$, set
$\xi(d)=\frac{1}{3}\left[1-(2d+1)^{-1/2}\right]$ 
and $\psi(d)=\frac{1}{2}\left[1-\exp\left(-\frac{d}{2}(1+e^{-d})\right)\right]$.
The next two assertions are our main points.
\begin{theorem}\label{thm.nu} If $d=\Theta(1)$, then w.h.p.
\[
\nu(G)>(1-o(1))\xi(d)m.
\]
\end{theorem}

\begin{theorem}\label{thm.tau} If $d=\Theta(1)$, then w.h.p.
\[
\tau(G)<(1+o(1))\psi(d)m.
\]
\end{theorem}
\nin
The proof of Theorem~\ref{MT} for fixed $d\geq 1/2$ is then completed by the following calculation.
\begin{lemma}\label{C} 
For any $d\geq 1/2$, $\psi(d)< 2\xi(d).$
\end{lemma}

A verification of the elementary (but not easy) Lemma~\ref{C} is sketched
in Appendix~\ref{Sec.C}.
Note the lemma is trivial for large enough $d$.  
It is actually true for \emph{all} positive $d$, but its already annoying proof becomes even more annoying
for $d$ below $1/2$ and, not needing this, we skip it.
We provide (and possess) no insight suggesting that the lemma is more than a lucky coincidence.
(It is sometimes just barely true; see Figure~\ref{Figure1} in the Appendix.)
On the other hand, we haven't much reason
to think that $\nu$ isn't significantly larger than what we're able to show.
(We \emph{guess} Theorem~\ref{thm.tau}, though slightly improvable, is 
close to the truth.)

Finally, completing the picture, we observe that for larger $d$, both 
$\nu$ and $\tau$ behave as one would expect.
Here we recall that
\beq{naive}
\mbox{for any $H$, $~\nu(H)\leq |H|/3~$ and $~\tau(H)< |H|/2$.}
\enq
(The bound on $\nu$ is trivial and that on $\tau$ is the standard observation 
that on average, for a random equipartition $V(H) = X\cup Y$, more than half the edges
of $H$ have ends in both $X$ and $Y$.)
It turns out that as $d\ra \infty$, both these bounds are (w.h.p.) asymptotically tight for $G$.
For $\tau$ this is due to Frankl and R\"odl~ \cite{FR} (see also \cite[Theorem 8.14]{JLR}; 
here, of course, it is just context, not part of the proof of Theorem~\ref{MT}).
We will show:
\begin{theorem}\label{PP}
If $d\gg 1 $ then w.h.p.
$~
\nu(G)\sim m/3.
$
\end{theorem}
\nin
This is an easy consequence of Pippenger's Theorem (or a slight variant thereof; see Section~\ref{SP}),
but despite some past interest (again, see Section~\ref{SP}), seems not to have
been pointed out previously.

\mn
\textbf{Outline.}
Section~\ref{Basics} gives definitions (mostly involving ``triangle-trees''),
proves a few simple results concerning these, and recalls a little standard machinery.
Section~\ref{Sec.BP} introduces the breadth-first triangle-trees $S^*(xy)$
and establishes couplings---Corollary~\ref{SSd} in particular---that underlie
Theorems~\ref{thm.smalld}-\ref{thm.tau}.
The latter are proved in Sections~\ref{PT1.2}-\ref{PTtau}
respectively, and, as noted above,
the proofs of Theorem~\ref{PP} and Lemma~\ref{C} are given
in Section~\ref{SP} and Appendix~\ref{Sec.C}.

\section{Basics}\label{Basics}

\subsection{Definitions}\label{Defs}
We use $H$ for a general graph, reserving $G$ for $G_{n,p}$ and $V$ for $V(G)$ ($=[n]$).
As is common, $H[X]$ is the subgraph of $H$ induced by $X$, $N(x)$ is the neighborhood of $x$
(in the graph under discussion), $N(x,y)=N(x)\cap N(y)$,
and $\nabla(A,B)$ is the set of edges joining disjoint sets of vertices $A,$ $B$.
We tend to think of 
graphs as edge sets, and in particular write $|H|$ for $|E(H)|$.

We use $\cR(H)$ for the set of triangles in $H$ and, for $\cA \subseteq \cR(H)$, 
$V(\cA)$ and $E(\cA)$ for the sets of vertices and edges in triangles of $\cA$.

A graph $T$ is a \textit{triangle-tree} if it can be gotten by starting with an edge $\rho$ (the \textit{root} of $T$) and repeatedly adding a triangle consisting of an already used edge and a not-previously-used vertex. A triangle-tree $T$ is a \textit{triangle-path} if each new triangle
uses an edge that was added in the preceding step. 

We will sometimes use simply ``tree'' and ``path'' for ``triangle-tree'' and ``triangle-path,''
since, with one tiny exception 
(in the proof of Proposition~\ref{small.d.Sd}), these are the only trees and paths we will see.

The \textit{length} of a (triangle-)path is its number of triangles. 
The \emph{distance}, $\dist(a,b)$ between
\emph{elements} $a,b$ of $H$ (meaning members of $V(H) \cup E(H) \cup \cR(H)$) 
is the length of a shortest path joining $a$ and $b$ (so $\dist(a,b)=\infty$ 
if there is no such a path, but we won't need this).
The \emph{depth} of an element of a tree is its distance to the root, and 
the depth of the tree itself is the largest of the depths of its elements.

For $A, B \in \cR(T)$ ($T$ a tree), we say $A$ is a \textit{child} of $B$ 
(and $B$ is \textit{the parent} of $A$) if $B$ is the triangle immediately preceding $A$
on the (unique) path from the root to $A$.
Similarly, the \emph{base} of an element $a$ of $T$ is the last \emph{edge}
preceding $a$ on the path joining the root to $a$
(so the edge that the first triangle containing $a$ shares
with the triangle that preceded it).

We say $H$ is \textit{triangle-connected} iff any two of its edges are connected
by a triangle-path.
In particular, a single edge is a triangle-path and triangle-connected. A 
\textit{triangle-component} of $H$ is a maximal triangle-connected subgraph of $H$, and
is \textit{trivial} if it is a single edge.
For $xy \in H$, we use $S(xy)$ for the triangle-component containing $xy$.

We may build any triangle-connected $S$ (say rooted at $e=xy$) by starting with $e$ and repeatedly 
adding triangles, each sharing at least an edge with what we already have.  Thus the number of 
added vertices (i.e.\ other than $x,y$) is at most half the number of added edges, with
\beq{eq}
\mbox{equality iff $S$ is a (triangle-)tree.}
\enq

We use $S^d$ for the Galton-Watson-like (random, possibly infinite) triangle-tree gotten by
starting with a root edge $\rho$
and letting each edge in turn give birth to a random number of 
triangles with distribution $\Po(d)$ (these choices made independently).
This ``ideal'' tree is susceptible to exact analysis, and the proofs of 
Theorems~\ref{thm.smalld}-\ref{thm.tau} will involve comparing it with
$S(xy)$.
We will ``interpolate'' between them using a breadth-first triangle tree, $S^*(xy)$, to
be defined in Section~\ref{Sec.BP}.  

Finally, we set
$S_\gc(xy)=\{e \in S(xy): \dist(xy,e)\le \gc\}$ and define $S^d_\gc$ analogously
(and similarly for $S^*_\gc(xy)$ when we get there).

\subsection{Small Claims}

\begin{proposition}\label{tree}
If $T$ is a finite triangle-tree, then $\tau(T)=\nu(T)$.
\end{proposition}

\begin{proof}
We proceed by induction on $|T|$. The statement is trivial if the depth,
say $t$, of $T$ is 0 or 1; so assume $t \ge 2$
and let $A$ be a triangle of depth $t$.
Let $e,f,g$ be the edges of $A$, with $e$ its base (so $A$ is the only triangle containing
either of $f,g$); let $B$ be the parent of $A$; and let $e,h,k$ be the edges of $B$,
with $h$ its base.

Let $\U$ be the set of triangles with base $e$ and notice that
$T\sm E(\U)$ is the edge-disjoint union of two triangle-trees,  $T_1$ and $T_2$
(one of them, rooted at $k$, of depth at most 1).
But then if $\m_i$ and $C_i$ are (resp.) a maximum matching and minimum cover of $T_i$,
induction gives $|\m_i|=|C_i|$ for each $i$, so
$\m:=\m_1\cup\m_2\cup \{A\}$ and $C:=C_1\cup C_2\cup \{e\}$ are a matching 
and cover of $T$ with $|\m|=|C|$, and the proposition follows.
\end{proof}

In what follows we will be interested in trees that are not too deep, say of depth $\gc=\gc(n)$ 
satisfying
\beq{gc}
\gc \ll \log n/\log \log n.
\enq
\begin{proposition}\label{small.tree}
For fixed d, $\gc$ as in \eqref{gc}, and distinct $x,y\in V$,
the probability that $xy\in G$ and
$S_\gamma(xy)$ is not a tree is less than $n^{-1+o(1)}$.
\end{proposition}
\begin{proof}
We first observe that, assuming $xy\in G$, if $S:=S_\gc(xy)$ is not a tree, then 
it contains, for some $i<2\gc$, a subgraph $T$ with $xy\in T$, 
$|V(T)\sm \{x,y\}| =i$ and $|E(T)\sm \{xy\}|\geq 2i+1$.
For if
\[
t=\min\{\ga: S_{\ga}(xy) \mbox{ is not a triangle-tree}\} ~(\leq \gc),
\]
then $S_{t}(xy)$ contains, for some vertex $v$, distinct triangle-paths $P_1$ and $P_2$ 
of length at most $t$ from $xy$ to $v$, and we have
$i:=|V(P_1\cup P_2)\sm \{x,y\}|\leq 2t-1$ and $|E(P_1\cup P_2)\sm \{xy\}|\geq 2i+1$
(since $P_1\cup P_2$ is triangle-connected and 
not a tree; see \eqref{eq}).

But the probability that $G$ contains such a $T$ is less than
\[
p\sum_{i\le 2\gamma-1} {n \choose i} {{i+2 \choose 2} \choose 2i+1} p^{2i+1}
=(O(\gc))^\gc n^{-1}
\]
(with the initial $p$ for $xy\in G$), which is $n^{-1+o(1)}$ for $\gc$ as in \eqref{gc}.
\end{proof}

\begin{proposition}\label{small.d.Sd} 
For the random triangle-tree $S^d$:
\begin{enumerate}
\item [(a)]
$S^d$ is finite with probability $1$ iff $d \le 1/2$;
\item [(b)]
the expected number of triangles of depth $i$ in $S^d$ is $(2d)^i/2$.
\end{enumerate}
\end{proposition}

\begin{proof}
These are basic properties of a Galton-Watson (GW) process (e.g.\ \cite[Section 5.1]{LP}).
We may associate with $S^d$ the (ordinary) random tree
$U$ with $V(U)=E(S^d)$ and $f$ a child of $e$ in $U$ iff, in $S^d$,
$e$ is the base of 
$f$.
Then $U$ is a GW tree with 
the number of children of each $e$ 
distributed as $L=2 \Po(d)$.
The assertions (a) and (b) are then given by Propositions~5.4 and 5.5 of \cite{LP},
which say (resp.) that $U$ is finite with probability 1 iff $\E L\leq 1$
(unless $L\equiv 1$), and that the expected number of vertices at depth $i$ in $U$ is $(2d)^i$.
\end{proof}

\subsection{Concentration}
 
We need two standard concentration facts (for the first see e.g.\ \cite[Theorem 2.1]{JLR}).
\begin{theorem}
\label{T2.1}
If $\xi $ is binomial with $\mathbb{E} \xi  = \mu $, then for $t \geq 0$,
\begin{align}
\Pr(\xi  \geq \mu + t) &\leq
\exp\left[-t^2/(2(\mu+t/3))\right], \nonumber\\
\Pr(\xi  \leq \mu - t) &\leq \exp[-t^2/(2\mu)].\nonumber
\end{align}
\end{theorem}

The second fact is ``McDiarmid's Inequality.''
(It is also called, for example, the Hoeffding-Azuma Inequality.  It is not the best one can say in
the situations below, but is enough for our purposes.)

\begin{theorem}[\cite{Mc}, Lemma 1.2] 
Let $X_1, \ldots, X_l$ be independent random variables, with $X_k\in A_k$ for each $k$. Suppose the (measurable) function $f: \prod A_k \rightarrow \mathbb R$ satisfies, for each k,
\beq{Mc.cond}
|f(X)-f(X')|\le c_k
\enq
whenever $X=(X_i:i\in [l])$ and $X'=(X'_i:i\in [l])$ differ only in their $k$th coordinates. 

Then for any $t>0$,
\[
\mbox{$\pr(|f-\E f| \ge t) \le 2 \exp[-2t^2/\sum c_k^2].$}
\]
\end{theorem}
\nin
We will always use this with $l=\C{n}{2}$ and
$X_i=\one_{\{e_i\in G\}}$, where $E(K_n)= \{e_i:i\in [l]\}$ (so $X=G$), in which case 
we have
\beq{martingale} 
\mbox{ if $f$ is Lipschitz (i.e.\ satisfies \eqref{Mc.cond} with $c_k=1$ $\forall k$) and 
$\E f\gg n$, then $f\sim \mathbb E f$ w.h.p.} 
\enq

\subsection{Number of subgraphs} 

For a graph $H$, let $\rho(H)=|E(H)|/|V(H)|$ (the \textit{density} of $H$),
and say $H$ is \textit{balanced} if every $H' \subseteq H$ has $\rho(H') \le \rho(H)$.

\begin{theorem}[\cite{AS}, Theorem 4.4.4] \label{2nd.m.m}
Let $H$ be balanced with $v$ vertices, $e$ edges and $a$ automorphisms, and 
let $X$ be the number of copies of $H$ in $G_{n,p}$. If $p \gg n^{-v/e}$ then w.h.p.
\[X \sim n^vp^e/a.\]
\end{theorem}

\subsection{Binomial v.\ Poisson}
For our limited purposes
we use simply $\|X-Y\|$ for the \textit{total variation distance} between discrete random variables $X$ and $Y$;
this is (by definition) half the $\l_1$ distance between their distributions, 
and is the minimum of $\pr(X\neq Y)$ under couplings of $X$ and $Y$.

\begin{proposition}\label{BPC}
For $n$ and $c\geq -n$ integers, $p\in [0,1]$, 
$X \sim \bin(n+c,p)$ and $Y \sim \Po(np)$,
\[\|X-Y\| \leq |c|p+O(p).\]
\end{proposition}

\begin{proof}
Let $Z \sim \bin(n,p)$. Since $\|X-Y\|  \le \|X-Z\|  + \|Z-Y\| $, and 
$\|Z-Y\| =O(p)$ (see (1.5) of~\cite{Vervaat} for a precise statement), 
it is enough to show $\|X-Z\| \leq|c|p$. To see this, we may couple $X$ and $Z$ 
by setting $l=\max\{n,n+c\}$, 
letting $\xi_1\dots \xi_l$ be independent with $\xi_i\sim \textrm{Ber}(p)$ ($\forall i$),
and setting $Z=\sum_{i\leq n} \xi_i$ and $X=\sum_{i\leq n+c}\xi_i$, 
yielding $\pr(X\neq Z) \leq |c|p$.\end{proof}

\section{Breadth first and branching}\label{Sec.BP}

We use $d(x)$ for the degree of $x$ (in $G$) and
$xyz$ for the triangle with vertices $x,y,z$, 
and assume in this section
that $d$ ($=(n-2)p^2$) $=\Theta(1)$.

As mentioned earlier,
the proofs of Theorems~\ref{thm.smalld}-\ref{thm.tau} depend on linking $S(xy)$ with 
the ideal triangle tree $S^d$ defined at the end of Section~\ref{Defs}, 
a connection based on comparing each of these with
the \emph{breadth-first} triangle-tree rooted at $xy\in G$.
This is the $xy$-rooted triangle-tree gotten by
\emph{processing edges in the order in which they enter the tree,}
where \emph{processing} $uv$ means adding all triangles $uvw$ with $w$ a vertex not yet in the tree;
more formally:

Fix an order $\prec$ on $V$, set 
$P_0=\{x,y\}$, and let $\R_0$ be the set of triangles on $xy$, $V_0=V(\R_0)$ 
and $Q_0= V_0\sm P_0$ ($ = N(x,y)$).
We process vertices $v_1,\ldots$ 
(this processing defined below), producing a
sequence 
$(\R_i,P_i,Q_i, X_i)$.  Each $\R_i$ will be the set of triangles of a tree,
with $V_i:=V(\R_i) =P_i  \amalg  Q_i$ and 
$X_i:=V\sm V_i$. When we finish 
processing $v_i$, vertices of
$P_i$ have been processed and vertices of $Q_i$ are ``in the queue'' (in the tree and 
waiting to be processed).
Of course we stop when the queue is empty, producing $S^*(xy)$.

We process vertices in the order in which they enter the evolving $V_i$, breaking ties according to $\prec$.
A key property that will hold throughout the evolution 
(which basically says we are building a tree) is
\beq{unique}
\mbox{each $v\in Q_i$ lies in a unique triangle, $T$, of $\R_i$, and $V(T)\sm \{v\}\sub P_i$.}
\enq

\emph{Processing} $v=v_i$ ($\in Q_{i-1}$) means:  with $ab$ the base of $v$ we form $\R_i$
by adding to $\R_{i-1}$ all triangles $avw$ and $bvw$ with 
\beq{w}
w\in X_{i-1}.
\enq
So we may---this will be natural below---also think of this as processing the \emph{edges} 
$av$ and $bv$, and each $w$ as above enters $Q_i$ with base one of $av$, $bv$.
Note there is no ambiguity here:
\[
N(a,b)\cap X_{i-1} = \0,
\]
since any $u\in N(a,b)$ not already in the tree when $ab$ was processed would have
been added to the tree at that time.
We then (in addition to $\R_i$) update $P_i=P_{i-1}\cup \{v_i\}$
and define $V_i$, $Q_i$, $X_i$ as above
(so $Q_i$ is $Q_{i-1}\sm\{v_i\}$ plus the $w$'s added at \eqref{w}).

Notice that this supports \eqref{unique}, which is true when $w$ enters the tree, and remains so
until $w$ is processed and removed from $Q_i$ (since none of the intervening steps
involves edges at $w$).

\begin{proposition}\label{S*S}
For any $\gc$, if
$S_{\gc}(xy)$ is a triangle-tree, then $S_\gc^*(xy)=S_\gc(xy)$.
\end{proposition}
\begin{proof}
Set $S=S(xy)$ and $S^*=S^*(xy)$.
Notice to begin that 
if $e,f$ are edges of $S^*$ with $e$
processed before $f$, then,
with $\dist^*$ denoting distance in $S^*$,
$\dist^*(xy,e)\leq \dist^*(xy,f)$
(by induction:  the base, $e'$, of $e$ was
processed no later than the base, $f'$, of $f$, so
$\dist^*(xy,e)= \dist^*(xy,e')+1\leq \dist^*(xy,f')+1=\dist^*(xy,f)$).

Suppose the proposition fails and let $A=uvw\in \R(S_\gc)\sm \R(S_\gc^*)$ with
$\dist(xy,A)$ minimum and $e=uv$ the base of $A$ in 
the unique path $P$ from $xy$ to $A$ in $S_\gc$.  
Then $\R(P)\sm \{A\}\sub \R(S_\gc^*)$ implies
$e$ was processed 
in the construction of $S^*$
and $A$ was not added, so $w$ 
was added before $e$ was processed.
But then, by the observation 
in the last paragraph,
$\dist^*(xy,w)\leq \dist^*(xy,e)+1\leq \gc$.
Thus $S_\gc$ contains two distinct paths from $xy$ to $w$, contradicting the assumption that
$S_\gc$ is a tree.
\end{proof}

For Proposition~\ref{S*Sd} and Corollary~\ref{SSd} we assume $\gc$ is as in \eqref{gc}.

\begin{proposition}\label{S*Sd}
On $\{xy\in G\}$, we may couple $S_\gc^*(xy)$ and $S_\gc^d$ so that they are equal w.h.p.
\end{proposition}
\nin
Combining this with Propositions~\ref{small.tree} and \ref{S*S} gives our main point:
\begin{corollary}\label{SSd}
On $\{xy\in G\}$, we may couple $S_\gc(xy)$ and $S_\gc^d$ so that they are equal w.h.p.
\end{corollary}

\begin{proof}[Proof of Proposition~\ref{S*Sd}]

We think of generating $S^*:=S^*(xy)\sub G$ 
by exposing edges as needed, where ``exposing'' an edge is 
deciding whether it's in $G$.
Precisely: we expose $\nabla(\{x,y\},V\sm\{x,y\})$, 
thus specifying the triangles of $S^*$ containing $xy$,
and then, for $i=1,\ldots$, $\nabla(v_i,X_{i-1})$, determining
the triangles added in the processing of $v_i$.

The number of triangles on $xy$ has law $\bin(n-2,p^2)$.
When we process $v:= v_i$, say with base $ab$, the number of triangles
added on $av$ (and similarly for $bv$) has law $\bin( |N(a)\cap X_{i-1}|,p)$
(note we do know $N(a)\cap X_{i-1}$ at this point),
which will usually be close to $\Po(d)$, since $|N(a)\cap X_{i-1}|$ is usually close to $np$.

We may think of a parallel generation of $S^d$:
when processing an edge $e$ in the generation of $S^*$, we simultaneously specify the 
number of triangles on $e$ in $S^d$, coupling so that the numbers of triangles in these 
two choices agree as often as possible.  Once the numbers agree, we may couple
so the trees themselves do as well.  
Of course this only makes sense as long as the trees agree:
if and when they do not, the coupling has failed and we lose interest.

It remains to bound the probability that the coupling fails. 
Set (with plenty of room) $\kappa = n^{0.1}$ and $\vs =n^{-0.2}$,
and define events $Q_1=\{|S^d_\gc|> \kappa\}$ and
$Q_2=\{\exists v\in V ~ d(v) \neq (1\pm \vs)np\}$.
Proposition~\ref{small.d.Sd}(b) (with Markov's Inequality) and Theorem~\ref{T2.1} imply
$\pr(Q_1) = o(1)$ and $\pr(Q_2) = n\exp[-\gO(\vs^2np)] ~(=o(1))$.

For the coupling, we use Proposition~\ref{BPC},
noting to begin that it
bounds the probability of failure when we process $xy$ by 
$\|\bin(n-2,p^2),\Po(d)\|  = O(p^2)$.

Suppose we have successfully coupled through the processing of $v_{i-1}$ 
and let $ab$ be the base of $v_i$.  The probability that the coupling now fails at (e.g.) $av_i$ is
at most 
\[
O(p[| |N(a)\cap X_{i-1}|-(n-2)p|+1]),
\] 
which is $O((\vs np+\kappa)p) = O(n^{-0.2})$ provided $d(a)=(1\pm \vs)np$ 
and $|V_{i-1}| <\kappa$.
Thus the (overall) probability that the coupling fails is at most
\[
\pr(Q_1)+\pr(Q_2) + O(p^2+\kappa n^{-0.2}) =o(1).\qedhere
\]

\end{proof}

\section{Proof of Theorem~\ref{thm.smalld}}
\label{PT1.2}

Assume first that $d=\Omega(1)$.
Since $d\leq 1/2$, 
Proposition \ref{small.d.Sd}(a) implies 
that on $\{xy\in G\}$, under the coupling of Corollary~\ref{SSd},
$S(xy)= S^d$ w.h.p.
(namely, $S(xy)= S^d$ if $S_\gc(xy)=S^d_\gc$ and 
the depth of $S^d$ is less than $\gc$, each of which is true w.h.p.).
Since $S^d$ is a tree, 
this implies that the expected number of edges (of $G$)
in triangle components that are not trees is $o(m)$,
so the \emph{actual} number is $o(m)$ w.h.p.

Now let $G_i$ run over the triangle components of $G$ and notice that, trivially,
\[
\nu(G) =\sum \nu(G_i),
\]
and similarly for $\tau$.  So, letting $\sum'$ denote sum only over $G_i$'s that are trees, and 
recalling Proposition~\ref{tree}, we have (w.h.p.)
\[
\mbox{$\tau(G) =\sum'\tau(G_i) +o(m) = \sum'\nu(G_i) +o(m) =\nu(G)+o(m)$.}
\]
This gives Theorem \ref{thm.smalld} when combined with 
\beq{iso.triangles}  
\mbox{$\nu(G)= \gO(m)$ w.h.p.}
\enq    

\begin{proof}[Proof of \eqref{iso.triangles}]
By \eqref{martingale} it's enough to show 
$\E\nu(G)= \gO(m)$.  But $\nu(G)$ is at least the number of isolated triangles in $G$
(an \emph{isolated} triangle being one sharing no edges with other triangles),
and the expected number of these is exactly
\[
\C{n}{3}p^3 (1-3p^2+2p^3)^{n-3} = \gO(m).
\]
\nin(Of course this---with the asymptotics of the number of
isolated triangles---could also be read off from the coupling with $S^d$.)\end{proof}

Now suppose $d\ll 1$ (i.e.\ $p\ll n^{-1/2}$).  Let $Y$ be the number of triangles in $G$,
$Y'$ the number of non-isolated triangles and $X$ the number of edges that lie in exactly one triangle.
For Theorem~\ref{thm.smalld}
it is enough to show that w.h.p.\ $Y'\ll Y$ (i.e.\ almost all triangles are isolated).

For $p\ll n^{-4/5}$ we just observe that $Y' =0$ w.h.p., since the expected 
number of pairs of triangles sharing an edge is $O(n^4p^5)$.
For larger $p$ ($p\gg 1/n$ is enough here), notice that 
$X\leq 3Y-Y'$.
From Theorem~\ref{2nd.m.m} we have
\beq{YEY}
\mbox{w.h.p.
$~Y\sim \E Y = \C{n}{3}p^3\sim n^3p^3/6,$}
\enq
while $p\ll n^{-1/2}$ gives
$
\E X = \C{n}{2}p(n-2)p^2(1-p^2)^{n-3} \sim n^3p^3/2\sim 3\E Y.
$
But then $\E Y' \leq 3\E Y-\E X \ll \E Y$ implies $Y'\ll \E Y$ w.h.p., which with 
\eqref{YEY} gives $Y'\ll Y$ w.h.p.

\section{Proof of Theorem~\ref{thm.nu}}

Given a graph $H$ and $w:\R(H)\ra [0,1]$
(values of $w$ will always be called \emph{weights}), 
we use $\cM^*_w$ for the greedy (triangle-)matching corresponding to $w$; namely: we
consider triangles in (increasing) order of their weights, and at each step 
add the triangle under consideration to $\cM^*_w$ iff it shares no edge with any triangle
already in the matching.
In particular when 
\beq{indw}
\mbox{$w$ is uniform from $[0,1]^{\R(H)}$,}
\enq
$\m^*:=\m^*_w$ is the usual
random greedy matching of $H$.
(Strictly speaking we have defined $\m^*_w$ only when the weights are distinct;
but for $w$ as in \eqref{indw}, this is true with probability 1
and will not be a concern.)

We will show that for any $x,y \in V$,
\beq{MP} 
\pr(xy \notin E(\cM^*) |~xy \in G)\ra (2d+1)^{-1/2},
\enq
where $\pr$ refers to the choices of $G$ and $w$.
This implies $\E\nu(G) >(1-o(1))(2d+1)^{-1/2}m$, which  with 
\eqref{martingale} gives Theorem~\ref{thm.nu}.
The proof of \eqref{MP}, which is inspired by \cite{Spencer}, is based on the connection with 
$S^d$ in Corollary~\ref{SSd}.
We need a few simple notions and observations.

\mn 

For a finite triangle-tree $T$ we work with the following recursive \emph{survival rule} for edges,
in which we may evaluate edges in any order for which each edge appears earlier than its base
(further specification of the order doesn't affect the outcome),
and ``dies'' means fails to survive:
\beq{rule}
\mbox{$e $ dies iff it is the base of a triangle whose other two edges survive.}
\enq
(For example, any edge that is the base of no triangle 
survives.)

\mn

For a general graph $H$, $e\in H$ and $w$ as in \eqref{indw}, 
let
\[
T(e)=\{f \in H: \mbox{there is a triangle-path from $e$ to $f$ on which the weights 
of the triangles decrease}\}
\]
(a random subgraph of $S(e)$).

It is easy to see that if $T(e)$ is a tree then $e$ is covered by $\m^*$ 
iff it dies when we apply \eqref{rule} to $T(e)$.
(In this case it's natural to think of evaluating edges in increasing order of their weights.
The present survival rule is the same as that of \cite{Spencer} applied to the (3-uniform) hypertree
with vertices $E(T)$ and edges $\R(T)$ (and the natural incidences).)

When $H=S^d$ we use $T^d$ for $T(\rho)$ (recall $\rho$ is the root of $S^d$).
As for \eqref{MP}, when we speak of $T^d$
(in Proposition~\ref{finite} and Lemma~\ref{survival}), ``probability'' refers to
the choices of both $S^d$ and $w$.

\begin{proposition}\label{finite}
$T^d$ is finite with probability $1$.
\end{proposition}

\begin{proof}
By Proposition~\ref{small.d.Sd}(b), the expected number of triangles of depth $i$ in $T^d$ is $(2d)^i/(2i!)$, which tends to zero as $i \rightarrow \infty$.
\end{proof}
Proposition~\ref{finite} and Corollary~\ref{SSd} imply
\[
\mbox{on $\{xy\in G\}$ we may couple $T(xy)$ and $T^d$ to agree w.h.p.}
\]
(namely, we can couple so $T(xy)=T^d$ whenever $S_\gc(xy)=S_\gc^d$ and the 
depth of $T^d$ is less than $\gc$).

In view of the preceding comments, this says that the probability in \eqref{MP}
tends to the probability that the root survives in $T^d$; so the proof of \eqref{MP}
(and Theorem~\ref{thm.nu}) is completed by the following calculation.

\begin{lemma}\label{survival}
Under \eqref{rule} the root of $T^d$ survives with probability
$(2d+1)^{-1/2}$.
\end{lemma}

\begin{proof}

It will be convenient to extend $w$ to edges:  set $w(\rho) =1$,
and for any other $e\in E(T^d)$ let $w(e)$ be the weight of the (unique) triangle on $e$ with minimum
depth. 

Let $f(x)$ be the probability that an edge of weight $x$ survives. 
Trivially, $f(0)=1$. The survival rule \eqref{rule}
says that an edge $e$ dies iff there is a child (triangle) of $e$ in $S^d$, 
say with edges $e,j,k$, such that
\beq{x.survive} 
\mbox{($w(j)=$) $~ w(k) < w(e)$ and both $j$ and $k$ survive.}
\enq
Given $w(e)$ (with $e,j,k\in S^d$ as above), the probability of \eqref{x.survive} 
is $\int_0^{w(e)} f^2(y)dy$, implying
\beq{f} f(x)=\sum_k\pr(Z=k)\left[ 1-\int_0^x f^2(y)dy\right]^k \enq
where $Z \sim \Po(d)$. Rewriting \eqref{f} with $F(x)=\int_0^x f^2(y)dy$ gives
\[
F'(x)=\left[ \sum_k\pr(Z=k)(1-F(x))^k\right]^2=e^{-2dF(x)}; \quad F(0)=0.
\]
The solution to this is
\[F(x)=\frac{1}{2d}\ln(2dx+1),\]
so we have
\[f(x)=F'(x)^{1/2}=(2dx+1)^{-1/2},\]
and the lemma follows.
\end{proof}

\section{Proof of Theorem \ref{thm.tau}}\label{PTtau}

For a partition $X \cup Y$ of $V=V(G)$ (we call each of $X, Y$ a \textit{block}), let $W=W(X,Y)=(W_0 \setminus W_1) \cup W_2$, where:

\begin{itemize}
\item $W_0= G[X] \cup G[Y]$;

\item $W_1=\{xy \in W_0: \mbox{all triangles on $xy$ are contained in the same block as $xy$}\}$;

\item $W_2=\{xy \in W_1: \mbox{there is a triangle $xyz$ with $xz, yz \in W_1$}\}$.

\end{itemize}

\nin It is easy to see that $W$ is a cover of $G$. 

For Theorem \ref{thm.tau}, again by \eqref{martingale},
it suffices to show that for $X\cup Y$ a uniformly random partition of $V$ (so each $v \in V$ is in $X$ with probability $1/2$, these choices made independently) and $W=W(X,Y)$,
\beq{not.in.S} \pr(xy \in W ~|~xy \in G) ~\ra~ \frac{1}{2}\left[1-\exp\left(-\frac{d}{2}(1+e^{-d})\right)\right].\enq

\begin{proof}[Proof of \eqref{not.in.S}]

\nin Set $Q=\{xy \in G\}$ and note to begin that 
\beq{ES1} 
\pr(xy \in W_0|Q)=1/2.
\enq
Set $p_k=e^{-d}d^k/k!$.  On $Q$ the distribution of the number of triangles on 
$xy$ is $\bin(n-2, p^2) \rad \Po(d)$, so
\beq{ES2} 
\pr(xy \in W_1|Q) \sim \sum_{k\ge 0} p_k2^{-k-1}=e^{-d/2}/2.
\enq
For $W_2$ we use Corollary~\ref{SSd}, now with $\gc=2$. 
Assigning vertices of $S^d_2$ to $X$ and $Y$ in the same way as 
vertices of $V$ (i.e.\ \emph{via} independent fair coin tosses), 
we may extend the coupling of the corollary to these choices so that
$xy\in W_2 \Leftrightarrow \rho\in W_2$ whenever $S_2(xy)=S^d_2$;
yielding
\beq{prxyW}
\pr(xy\in W_2|Q) = \pr(\rho\in W_2) + o(1)
\enq
(where as usual $o(1)$ can be negative).  On the other hand,
\beq{rhoW}
\pr(\rho \in W_2) =  \sum_{k\ge1} p_k2^{-k-1}(1-(1-\alpha)^k),
\enq
where 
\[
\alpha= \sum_{l,m \ge 0} p_lp_m2^{-l}2^{-m}=
\left[e^{-d}\sum_{l\ge0}\frac{d^l2^{-l}}{l!}\right]^2=e^{-d}\]
is the probability that, given $\{x,y,z\}\sub X$ (e.g.), 
all triangles containing either of $xz$, $yz$ also lie in $X$. 
Then rewriting the r.h.s.\ of \eqref{rhoW} as
\[
\frac{e^{-d}}{2}\left[e^{d/2}-\exp\left(\frac{d}{2}(1-e^{-d})\right)\right]
\]
and combining with \eqref{ES1}-\eqref{prxyW} gives \eqref{not.in.S}.
\end{proof}

\section{Proof of Theorem~\ref{PP}}\label{SP}

For $d>\log^{3+\eps} n$ Theorem~\ref{PP} was proved 
(somewhat implicitly) in \cite{FR85},
and, as observed in \cite{BDZ},
direct application of Pippenger's Theorem improves this to $d\gg \log n$,
where w.h.p.\ each edge of $G$ is in $(1+o(1))d$ triangles.
(Pippenger's Theorem was never published and first appeared in \cite{Furedi};
see also e.g. \cite[Theorem~4.7.1]{AS}.)

In fact Pippenger's Theorem can also be used 
to prove Theorem~\ref{PP}, but we will find it convenient to use the following variant,
a simplest instance of
\cite[Theorem 1.5]{JK}.  (For fractional things see e.g.\ \cite{SU}.)

For a hypergraph $\h$ and $\vp:\h\ra [0,1]$, let
\[
\ga(\vp) = \max\sum\{\vp(A):x,y\in A\in\h\},
\]
the max over distinct vertices $x,y$ of $\h$.

\begin{theorem}\label{TJK}
For fixed r, if $\h$ is r-uniform and $\vp:\h\ra [0,1]$ is a fractional matching,
then 
\[
\nu(\h) > (1-o(1))\sum_{A\in \h}\vp(A),
\]
where $o(1)\ra 0$ as $\ga(\vp)\ra 0$.
\end{theorem}
\nin
(The statement in \cite{JK} also assumes $\sum\vp(A)\ra\infty$, but this is easily seen to be unnecessary.)

We will (of course) use Theorem~\ref{TJK} with $\h=\R(G)$ (and $V(\h)=E(G)$).
Let $1\gg\vs \gg d^{-1/2}$ (recall $d\gg 1$) and $D=(1+\vs)d$.
Say $e\in E(G)$ is \emph{heavy} if it lies in at least $D$ triangles,
and define the fractional
matching $\vp:\h\ra [0,1]$ by
\[
\vp(A)=\left\{\begin{array}{ll}
1/D&\mbox{if $A$ contains no heavy edges,}\\
0&\mbox{otherwise.}
\end{array}\right.
\]
Of course $\ga(\vp)\ra 0$, so to get Theorem~\ref{PP} from Theorem~\ref{TJK} we just need
\[
\mbox{w.h.p. $\sum \vp(A)\sim m/3$;}
\]
this will follow from 
\beq{unheavy}
\mbox{w.h.p.\ the number of triangles of $G$ containing heavy edges is $o(n^3p^3)$.}
\enq
For if \eqref{unheavy} is true then, since $|\R(G)|\sim \C{n}{3}p^3$ w.h.p.
(see Theorem~\ref{2nd.m.m}) and $D\sim np^2$, we have
\[
\mbox{w.h.p. $~\sum\vp(A) \sim n^3p^3/(6D)\sim n^2p/6\sim m/3$.}
\]
Finally,
for $x,y\in V$,
Theorem~\ref{T2.1} bounds the probability that $xy$ (is in $G$ and) lies in at least 
$(1+\gc)d$ triangles by $p\exp[- \gc^2 d/[2(1+\gc/3)]$; so the expected 
number of triangles containing heavy edges is less than
\[
\C{n}{2}p\sum_{i\geq 0}\exp[- 2^{2i} \vs^2d/(2(1+2^i\vs/3))] (1+2^{i+1}\vs) d 
= o(n^3p^3),
\]
and Markov's Inequality then gives \eqref{unheavy}.

\mn
\textbf{Acknowledgment.}  We thank David Galvin for telling us the problem.

\appendix
\section{Proof of Lemma~\ref{C}}\label{Sec.C}

There is nothing very interesting here and we aim to be brief.
The lemma is easy 
when $d \ge 8$, since then $\xi(d)>1/4$ (while $\psi(d)<1/2$ for all $d$). 
For $d \in [1/2,8]$, we show the inequality in the form
\beq{equiv} 4(2d+1)^{-1/2}-3\exp(-\frac{d}{2}(1+e^{-d}))<1.\enq

Let
\[f(d)=4(2d+1)^{-1/2}; \quad g(d)=3e^{-d/2}; \quad h(d)=\exp(-\frac{d}{2}e^{-d}).\]
(So the l.h.s.\ of \eqref{equiv} is $f-gh$.) It is easy to see that 
\beq{hd} 
\mbox{$h(d)$ is decreasing on $[0,1]$ and increasing on $[1,\infty)$}
\enq
and 
\beq{fg.convex} 
\mbox{$f$ and $g$ are convex.}
\enq

\nin \textit{Case 1}:  $d \in [k, k+1]$ $(k=1,2, \ldots, 7)$

Fix $k$ and let $f_1(\cdot)$ be the line through $(k,f(k))$ and $(k+1,f(k+1))$, 
and $g_1(\cdot)$ the tangent to 
$g$ at $(k+1,g(k+1))$.  From \eqref{hd} and \eqref{fg.convex} we have
\[f(d)-g(d)h(d) \le f_1(d)-g_1(d)h(k),\]
so it suffices to show (for $d\in [k,k+1]$)
\beq{linear} 
f_1(d)-g_1(d)h(k)< 1.
\enq
But the l.h.s.\ of \eqref{linear} is a linear function of $d$, so it's enough to check
the inequality at the endpoints---which we won't, but 
for example, when $k=1$, the function is $Ad+B$, with
\[A=\frac{4}{5}\sqrt 5-\frac{4}{3}\sqrt 3+\frac{3}{2}e^{-(1/2e+1)} (\approx -0.06144) <0,\]
\[B=\frac{8}{3}\sqrt 3 - \frac{4}{5}\sqrt 5 -6e^{-(1/2e-1)} (\approx 0.99352) <1\]
(so \eqref{linear} holds). Other $k$'s are similar.

\nin \textit{Case 2}: $d \in [1/2,1]$

Here we take $f_1(\cdot)$ to be the line through $(1/2,f(1/2))$ and $(1,f(1))$, 
and $g_1(d)$ the tangent to $g$ at $(1,g(1))$, and again just need to show 
the analogue of \eqref{linear}, i.e.
\beq{linear'}
f_1(d)-g_1(d)h(1)<1.
\enq
The l.h.s.\ of this is $Ad+B$, with
\[
A=\frac{8}{3}\sqrt 3-4\sqrt 2+\frac{3}{2}e^{-1/2(1/e+1)} (\approx-0.28111) <0,
\]
\[
B=4\sqrt 2-\frac{4}{3}\sqrt 3-\frac{9}{2}e^{-1/2(1/e+1)} (\approx 1.07664),
\]
so is maximized (on $[1/2,1]$) at $d=1/2$, where it is strictly less than $1$.

\begin{figure}[h]
\includegraphics[width=6cm]{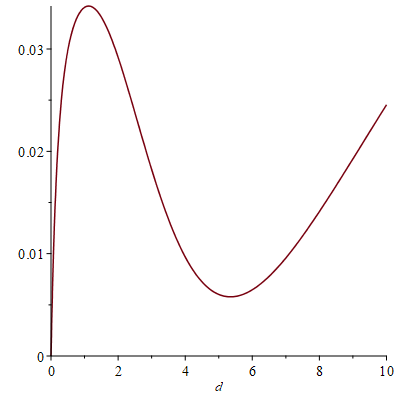}
\caption{$2\xi(d)-\psi(d)$ for $d \in [0,10]$}\label{Figure1}
\end{figure}


\begin{thebibliography}{AA}

\bibitem{AS}
N. Alon and J. Spencer,
\emph{The Probabilistic Method, 4th ed.,}
Wiley Series in Discrete Mathematics and Optimization,
John Wiley and Sons, Inc., 2016.

\bibitem{BCD}
P. Bennett, R. Cushman, and A. Dudek,
Closing the random graph gap in Tuza's Conjecture through the online triangle packing process,
preprint. 	arXiv:2007.04478 [math.CO]

\bibitem{BDZ}
P. Bennett, A. Dudek, and S. Zerbib,
{Large triangle packings and Tuza's conjecture in sparse random graphs,}
\emph{Combinatorics, Probability and Computing,}
to appear.

\bibitem{FR85} P.\ Frankl and V. R\"odl, 
Near perfect coverings in graphs and hypergraphs,
\emph{Eur.\ J.\ Comb.} \textbf{6} (1985), 317-326.

\bibitem{FR} P.\ Frankl and V. R\"odl, 
Large triangle-free subgraphs in graphs without $K_4$,
\emph{Graphs and Comb.} \textbf{2} (1986), 135-144.

\bibitem{Furedi}
Z.\ F\"uredi, Matchings and covers in hypergraphs,
\emph{Graphs and Comb.} \textbf{4} (1988), 115-206.

\bibitem{Galvin}
D.\ Galvin, personal communication.

\bibitem{Haxell}
P.\ Haxell,
Packing and covering triangles in graphs,
\emph{Discrete Math.} \textbf{195} (1999), 251-254.


\bibitem{HKT}
P.\ Haxell, A.\ Kostochka and S.\ Thomass\'e,
A stability theorem on fractional covering of triangles by edges,
\emph{European J.\ Comb.} \textbf{33} (2012), 799-806.


\bibitem{JLR} S. Janson, T. \L uczak and A. Ruci\'nski,
{\em Random Graphs}, Wiley, New York, 2000.

\bibitem{JK}
J. Kahn,
A linear programming perspective on the Frankl-R\"odl-Pippenger theorem,
\emph{Random Structures \& Algorithms} \textbf{8} (1996), 149-157.

\bibitem{LP}
R. Lyons and Y. Peres,
\emph{Probability on Trees and Networks,}
Cambridge Series in Statistical and Probabilistic  Mathematics, 42.
Combridge University Press, New York, 2016.


\bibitem{Mc}
C. McDiarmid,
On the method of bounded differences,
\emph{Surveys in Combinatorics} (1989), 148-188,
London Math. Soc. Lecture Note Ser., 141, Cambridge Univ. Press, Cambridge, 1989.


\bibitem{SU}
E.R.\ Scheinerman and D.H.\ Ullman, \emph{Fractional Graph Theory},
Wiley, New York, 1997.

\bibitem{Spencer}
J. Spencer, Asymptotic packing via a branching process,
\emph{Random Structures and Algorithms} \textbf{7} (1995), 167-172.

\bibitem{Tuza}
Zs.\ Tuza,
Conjecture, Finite and Infinite Sets,
Eger, Hungary 1981.
A. Hajnal, L. Lov\'asz, V.T. S\'os (eds.)
\emph{Proc. Colloq. Math. Soc. J. Bolyai},
vol. 37, pp. 888.
North-Holland, Amsterdam (1984).


\bibitem{Vervaat}
W.\ Vervaat,
{Upper bounds for the distance in total variation between the binomial or negative binomial and the Poisson distribution,}
\emph{Statistica Neerlandica} \textbf{23} (1969), 79-86.

\end{thebibliography}
\end{document}